\documentclass{aims}
\usepackage{amsmath}
  \usepackage{paralist}
  \usepackage{graphics} 
  \usepackage{epsfig} 
 \usepackage[colorlinks=true]{hyperref}
\hypersetup{urlcolor=blue, citecolor=red}

\def\currentvolume{X}
 \def\currentissue{X}
  \def\currentyear{200X}
   \def\currentmonth{XX}
    \def\ppages{X--XX}
	  \def\DOI{}

\newtheorem{theorem}{Theorem}

\newtheorem{proposition}[theorem]{Proposition}

\newtheorem{remark}[theorem]{Remark}

\def\be#1{\begin{equation}\label{#1}}
\def\ee{\end{equation}}
\def\req#1{{\rm(\ref{#1})}}
\def\xdag{x^\dagger}

\def\xad{{x_\alpha^\delta}}
\def\ydel{y^\delta}
\def\norm#1{\hspace*{0.2ex}\|#1\|}

\def\R{\mathbb{R}}

\title[Lavrentiev's regularization method]{Lavrentiev's regularization method in Hilbert spaces revisited}

\author[Bernd Hofmann and Barbara Kaltenbacher and Elena Resmerita]{}

\subjclass{Primary:  65J22,	65J22 ; Secondary: 45Q05, 35R30.}
 \keywords{Lavrentiev regularization, monotone operators, convergence analysis, source conditions, regularization parameter choice}

 \email{hofmannb@mathematik.tu-chemnitz.de}
 \email{barbara.kaltenbacher@aau.at}
 \email{elena.resmerita@aau.at}

\thanks{The first author was supported by DFG grant HO~1454/8-2.
The second author acknowledges support by the Austrian Science Fund FWF under grand I2271.
The second and third author were supported by the Karl Popper Kolleg ``Modeling-Simulation-Optimization'' of the AAU}

\begin{document}
\maketitle

\centerline{\scshape Bernd Hofmann}
\medskip
{\footnotesize
 \centerline{Technische Universit\"at Chemnitz}
   \centerline{Reichenhainer Str. 41}
   \centerline{09111 Chemnitz, Germany}
} 

\medskip

\centerline{\scshape Barbara Kaltenbacher and Elena Resmerita}
\medskip
{\footnotesize
 \centerline{Alpen-Adria-Universit\"at Klagenfurt}
   \centerline{Universit\"aatsstra\ss e 65-67}
   \centerline{9020 Klagenfurt, Austria}
}

\bigskip

 \centerline{(Communicated by the associate editor name)}


\begin{abstract}
In this paper, we deal with nonlinear ill-posed problems involving monotone operators and consider Lavrentiev's regularization method. This approach, in contrast to Tikhonov's regularization method,  does not make use of the adjoint of the derivative. There are plenty of qualitative and quantitative convergence results in the literature, both in Hilbert and Banach spaces. Our aim here is mainly to contribute to convergence rates results
 in Hilbert spaces based on some types of error estimates derived under various source conditions and to interpret them in some settings. In particular, we propose and investigate new variational source conditions adapted to these Lavrentiev-type techniques. Another focus of this paper is to exploit the concept of approximate source conditions.
\end{abstract}
\section{Introduction} \label{sec:intro}
For an infinite dimensional and separable real Hilbert space $X$ with norm $\|\cdot\|$ and inner product $\langle\cdot,\cdot\rangle$, we consider the ill-posed operator equation
\be{Fxy}
F(x)=y
\ee
with the nonlinear forward operator $F: \mathcal{D}(F) \subseteq X \to X$ and assume that only noisy data $\ydel \in X$ are available such that the deterministic noise model
\be{delta}
\norm{y-\ydel}\leq\delta
\ee
with noise level $\delta>0$ applies.  In this context, let $\xdag \in \mathcal{D}(F)$ denote an exact solution to equation (\ref{Fxy}) and $\bar{x} \in X$ a reference element (initial guess).
Throughout this work we suppose that $F$ is a monotone operator, i.e. we have
\be{Fmon}
\langle F(x)-F(\tilde{x}),x-\tilde{x}\rangle \geq 0 \qquad \mbox{for all} \quad x,\tilde{x}\in \mathcal{D}(F).
\ee

{Then the operator equation \req{Fxy} is well-posed if
$$\langle F(x)-F(\tilde{x}),x-\tilde{x}\rangle \geq \|x-\tilde x\|\,\theta(\|x-\tilde x\|) \quad \mbox{for all}\;\; x,\tilde{x}\in \mathcal{D}(F)$$
holds for some index function $\theta$. A function $\theta: (0,\infty) \to (0,\infty)$ is called index function if it is continuous, strictly
increasing and moreover satisfies the limit condition $\lim_{t \to +0} \theta(t)=0$.
}

If well-posedness of \req{Fxy} fails, a regularization approach is required in order to find stable approximate solutions to the ill-posed equation \req{Fxy}. We are going to construct regularized solutions
$\xad$ to $\xdag$ by solving the equation
\be{Lavrentiev}
F(\xad)+\alpha(\xad-\bar{x})=\ydel.
\ee
This approach, which was originally suggested in \cite{Lavrentiev67}, is often called Lavrentiev's regularization method. Alternatively, it is called perturbation method (cf.~\cite{LiuNash96}) in order
to refer to the manner of obtaining regularized solutions as singular perturbations.
It can be shown that  equation \eqref{Lavrentiev} has a unique solution $\xad$, under various types of assumptions. We refer below to two variants of such assumptions, namely (i) in the context of a globally monotone operator $F$ defined on the whole space $X$  and (ii) when the domain of $F$ might be smaller than $X$ and $F$ is locally monotone there.

(i) If $F: X \to X$ is globally monotone,  i.e.~\req{Fmon} is valid with $\mathcal{D}(F) =X$, and $F$ is a hemicontinuous operator then, by the Browder-Minty theorem, \req{Lavrentiev} has a unique solution $\xad$ for all $\bar{x} \in X$ and $\ydel \in X$.

(ii) If there is a ball $B_r(\xdag) \subset \mathcal{D}(F)$ around a solution $\xdag$ to equation \req{Fxy} with radius $r=\|\xdag-\bar x\|+\frac{\delta}{\alpha}$ and $F:\mathcal{D}(F) \subset X \to X$ is a Fr\'echet differentiable operator monotone and hence
hemicontinuous operator
in the ball, then \req{Lavrentiev} has a unique solution $\xad$ in $B_r(\xdag)$ for all $\bar{x} \in X$ and $\ydel \in X$ (cf.~\cite[Theorem~1.1]{Taut02}).

In both cases, $\xad$ depends continuously on $\ydel$.

{Note that we will work under  assumptions (i), for simplicity of exposition.}
A comprehensive study of Lavrentiev's regularization method for nonlinear equations in Hilbert spaces with monotone operators, even in a more general setting, can be found in the book \cite{AlbRya06}, for modifications of the method see also \cite{BakuSmi06,BakuSmi07,George13,Janno00}.
Nevertheless, for completeness of exposition and since some of the estimates will be needed later on, in Section~\ref{sec:conv} we provide a short summary of the arguments leading to convergence of $\xad$ from (\ref{Lavrentiev}) to $\xdag$  (see also \cite{Ramm00,BaGo94}).

As in the case of Tikhonov's regularization method (cf., e.g.,~\cite[Section~3.2]{SKHK12} and \cite{Hof15}), also for Lavrentiev's regularization method
a certain additional smoothness of $\xdag$ with respect to the forward operator is required in order to derive convergence rates.
At the origin of such studies, source conditions (range conditions) of H\"older-type (cf.,~e.g.,~\cite{Janno00,LiuNash96,Taut02, Taut04}) were under consideration, however attaining   here the specific form
\be{LavHoelder}
\xdag-\bar x \in \mathcal{R}(A^p), \quad 0<p \le 1.
\ee
Extensions of type
\be{Lavindex}
\xdag-\bar x \in \mathcal{R}(\varphi(A)),
\ee
with index functions $\varphi$, can also be found in the literature (cf.,~e.g., \cite{MaNa13} and \cite{Argyros13,Seme10}).
In this context, $A\in\mathcal{L}(X,X)$ denotes a bounded linear operator related to some  derivative $F^\prime(\xdag)$ of the forward operator at the solution point.
If these conditions exceed the present solution smoothness, some weaker forms of source conditions might be of interest.

Recent developments  regarding convergence rates theory involve variational inequalities as source conditions for Tikhonov regularization, as introduced by \cite{HKPS07}. One focus of this work is to investigate in case of Lavrentiev regularization not only
the already known   conditions in a Tikhonov-type context, but also new conditions adapted to the Lavrentiev's techniques  - see  Section ~\ref{sec:rates_var} in this respect. Another focus is to exploit the concept of approximate source conditions in the sense of \cite{HeinHof09,Hof06} in Section~\ref{sec:rates_approx}.
Finally,  we provide some illustrative examples in Section \ref{sec:examples}.

\section{Convergence} \label{sec:conv}

Similarly to the Tikhonov regularization for nonlinear operator equations in a Hilbert space setting (cf.~\cite[Sect.~10.2]{EHN96}, also for Lavrentiev regularized solutions $\xad$ convergence is global, in the sense that no sufficient closeness between $\bar{x}$ and $\xdag$
needs to be required. Also a general  \textcolor{blue}{condition} which is well-known from Tikhonov regularization occurs again here, namely we assume that $F$ is
weak-to-norm sequentially closed,
 which means that
\be{Fweakseqclosed}
x_n\rightharpoonup \tilde x\mbox{ and } F(x_n)
\to
z_0 \ \Rightarrow \tilde x\in\mathcal{D}(F) \mbox{ and } F(\tilde x)=z_0.
\ee
Monotonicity and hemicontinuity of $F$ are together sufficient for \eqref{Fweakseqclosed} in the case $\mathcal{D}(F)=X$ (cf., e.g., Lemma 1.4.5 and Theorem 1.4.6 in \cite{AlbRya06}).

We assume that the set $L$ of solutions to equation \req{Fxy} for given $y \in X$ is not empty. We  derive below formulas for handling the error of regularized solutions $\xad$ with respect to exact solutions $\xdag \in L$ in order to show convergence for appropriate choices of the regularization parameter $\alpha$, where in particular $\bar{x}$-minimum-norm solutions are targeted. As usual, we say that a solution $\xdag_{min} \in L$ is the $\bar{x}$-minimum-norm solution if
$\|\xdag_{min}-\bar{x}\|=\min\limits_{\xdag \in L}\|\xdag-\bar{x}\|$.

Under {condition} \req{Fweakseqclosed}, the set $L$ is weakly closed and hence such a minimum-norm solution $\xdag_{min}$ exists (see, e.g.,~\cite[Prop.~3.14]{SKHK12}). If in addition $L$ is convex and closed,  then $\xdag_{min}$ is uniquely determined. { This is actually the case throughout this paper, as the context (i) that we work with, namely $\mathcal{D}(F)=X$,  monotonicity and hemicontinuity of the operator, implies maximal monotonicity of $F$ and is sufficient  for the convexity and closedness of $L$ (cf. Th. 1.4.6 and Corollary 1.4.10 in \cite{AlbRya06}).}

Testing \eqref{Lavrentiev} with $\xad-\xdag$ and moreover with $F(\xad)-F(\xdag)$ leads to
\begin{eqnarray} \label{bas4}
&&\langle F(\xad)-F(\xdag), \xad-\xdag\rangle + \langle y-\ydel,\xad-\xdag\rangle \nonumber
\\&&+\alpha \norm{\xad-\xdag}^2+\alpha \langle \xdag-\bar{x},\xad-\xdag\rangle=0
\end{eqnarray}
and
\begin{eqnarray*}
&&\norm{F(\xad)-F(\xdag)}^2 + \langle y-\ydel,F(\xad)-F(\xdag)\rangle
\\&&+\alpha \langle F(\xad)-F(\xdag), \xad-\xdag\rangle+\alpha \langle \xdag-\bar{x},F(\xad)-F(\xdag)\rangle=0\,,
\end{eqnarray*}
respectively. By using the monotonicity condition \eqref{Fmon} of $F$ in combination with the Cauchy-Schwarz inequality this implies the following three basic estimates, which will be required below.
\begin{eqnarray}
&&\norm{\xad-\xdag}^2\leq \langle \xdag-\bar{x},\xdag-\xad\rangle+\frac{\delta}{\alpha}\norm{\xad-\xdag},\label{estx0}\\
&&\norm{\xad-\xdag}\leq \norm{\xdag-\bar{x}}+\frac{\delta}{\alpha},\label{estx}\\
&&\norm{F(\xad)-F(\xdag)}\leq \alpha\norm{\xdag-\bar{x}}+\delta.\label{estFx}
\end{eqnarray}
Taking into account these estimates one can conclude convergence with a priori or a posteriori parameter choice by standard  arguments of regularization theory.

\paragraph{\bf A priori parameter choice}

We choose $\alpha_*=\alpha_*(\delta)$ a priori such that it satisfies the limit conditions
\be{apriori}
\alpha_*(\delta)\to 0 \quad \mbox{and} \quad \frac{\delta}{\alpha_*(\delta)}\to 0 \quad \mbox{as} \quad \delta \to 0.
\ee
Inserting this into \req{estx} we obtain that $\norm{x_{\alpha_*(\delta)}^\delta-\xdag}$ is uniformly bounded as $\delta\to0$. Thus, for any sequence $\delta_n\to0$ there exists a weakly convergent subsequence $x_{\alpha_*(\delta_{n_k})}^{\delta_{n_k}}$ with weak limit $\tilde x$. On the other hand, \req{estFx} and \req{apriori} yield $F(x_{\alpha_*(\delta_n)}^{\delta_n})\to y$ and hence, from the weak sequential closedness \req{Fweakseqclosed}, we conclude that $\tilde x$ solves \req{Fxy}.  Thus, we have weak subsequential convergence to a solution. Convergence is even strong since by \req{estx0}, \req{estx}, which remain valid for any solution $\tilde x$ of \req{Fxy},
we have
\[
\begin{aligned}
&\limsup_{k\to\infty}\
\norm{x_{\alpha_*(\delta_{n_k})}^{\delta_{n_k}}-\tilde x}^2\\
&\leq \limsup_{k\to\infty}\left\{
\langle \tilde x-\bar{x},\tilde x-x_{\alpha_*(\delta_{n_k})}^{\delta_{n_k}}\rangle
+\frac{\delta_{n_k}}{\alpha_*(\delta_{n_k})} \norm{x_{\alpha_*(\delta_{n_k})}^{\delta_{n_k}}-\tilde x}\right\}
=0\,,
\end{aligned}
\]
where we have used  weak convergence,  \req{apriori} and uniform boundedness of $\norm{x_{\alpha_*(\delta_{n_k})}^{\delta_{n_k}}-\tilde x}$. Along the lines of the proof of \cite[Theorem~2.1.2]{AlbRya06}, we show that every such limit $\tilde x$ is the $\bar{x}$-minimum-norm solution. Let $\xdag$ be an arbitrary element of $L$.
From \req{estx0} we have
$$\langle \xdag-\bar{x},\xdag-x_{\alpha_*(\delta_{n_k})}^{\delta_{n_k}}\rangle+\frac{\delta_{n_k}}{\alpha_*(\delta_{n_k})}\norm{x_{\alpha_*(\delta_{n_k})}^{\delta_{n_k}}-\xdag} \ge 0 \qquad \mbox{for all}\quad \xdag \in L$$
and hence, under \req{apriori} with  $x_{\alpha_*(\delta_{n_k})}^{\delta_{n_k}}\to \tilde x \in L$ as $k \to \infty$, the following holds:
$$\langle \xdag-\bar{x},\xdag-\tilde x\rangle\ge 0 \qquad \mbox{for all}\quad \xdag \in L.$$
Due to the convexity of $L$, one has $x_t:=t\tilde x+(1-t)\xdag_{min} \in L$, for all $0 \le t <1$. By setting $\xdag:=x_t$
$$\langle x_t-\bar{x},(1-t)(\xdag_{min}-\tilde x)\rangle\ge 0\qquad \mbox{for all}\quad 0 \le t <1,$$
one has
$$\langle x_t-\bar{x},\xdag_{min}-\tilde x\rangle\ge 0$$
which yields, for the limit $t \to 1$,
$$\langle \tilde x-\bar{x},\xdag_{min}-\tilde x\rangle\ge 0 \quad \mbox{and}\quad \langle \tilde x-\bar{x},\xdag_{min}-\bar{x} \rangle \ge \|\tilde x-\bar{x}\|^2.$$
This, however, provides us with the inequality $\|\tilde x-\bar{x}\| \le \|\xdag_{min}-\bar{x}\|$ which shows that $\tilde x$ is the  $\bar{x}$-minimum-norm solution.
By a subsequence-subsequence argument, the whole sequence converges to $\xdag_{min}$. The discussion above proves the following proposition.

\begin{proposition}
For an a priori parameter choice $\alpha_*=\alpha_*(\delta)$ satisfying \req{apriori} and a sequence $\delta_n\to0$, the sequence of associated regularized solutions $x_{\alpha_*(\delta_{n})}^{\delta_{n}}$
converges  strongly to the $\bar{x}$-minimum-norm solution $\xdag_{min}$ of \req{Fxy}.
\end{proposition}

\paragraph{\bf A posteriori parameter choice: a variant of discrepancy principle}
By inspection of the detailed considerations in the previous paragraph concerning the convergence for a priori choices of the regularization, we immediately find the following extension to a posteriori parameter choices.

\begin{proposition} \label{pro:apo}
For an a posteriori parameter choice $\alpha_*=\alpha_*(\delta,y^\delta)$, well-defined for sufficiently small $\delta>0$ and satisfying the conditions
\be{aposteriori}
\frac{\delta}{\alpha_*(\delta,y^\delta)}\to 0 \quad \mbox{as} \quad \delta \to 0
\ee
as well as
\be{limity}
\|F(x_{\alpha_*(\delta)}^{\delta})-y\| \to 0 \quad \mbox{as} \quad \delta \to 0,
\ee
the sequence of regularized solutions $x_{\alpha_*(\delta_{n})}^{\delta_{n}}$ associated with a sequence $\delta_n\to 0$
converges  strongly to the $\bar{x}$-minimum-norm solution  $\xdag_{min}$ of \req{Fxy}.
\end{proposition}

Now we are going to apply Proposition~\ref{pro:apo} to a variant of the discrepancy principle (compare to the residual principle in Chapter 3 of \cite{AlbRya06}) and assume in this context that
\be{pospar}
\|F(\bar{x})-y\|>0, \quad \mbox{i.e.$\quad\bar{x}$ does not solve \req{Fxy}}.
\ee
Given a sufficiently small noise level $\delta>0$ and data $y^\delta \in X$ satisfying \req{delta}, we fix $0<\kappa<1,\;\tau>1$ and $\alpha_0>0$ such that
\be{alphanull}
\|F(x^\delta_{\alpha_0})-y^\delta\|>\tau \delta^\kappa.
\ee
Under \req{pospar}, the existence of such an $\alpha_0$ for sufficiently small $\delta>0$ is due to the limit condition
$$\lim \limits_{\alpha \to \infty} \|F(\xad)-\ydel\|=\|F(\bar{x})-\ydel\|$$
and $\|F(\bar{x})-\ydel\| \ge \|F(\bar{x})-y\|-\delta > \tau \delta^\kappa$ for $\delta<\min(1, \left(\frac{\|F(\bar{x})-y\|}{\tau+1}\right)^{1/\kappa}).$
Moreover, we know that there is some $\underline{\alpha}>0$ such that
$$\|F(x^\delta_{\alpha})-y^\delta\| \le \tau \delta^\kappa \quad \mbox{for all} \quad 0<\alpha \le \underline{\alpha}.$$
This is a consequence of \req{Lavrentiev},\req{estx} and the estimate
$$\|F(\xad)-\ydel\|=\alpha\|\xad-\bar{x}\| \le \alpha \left(\|\xad-\bar{x}\|+\|\xdag-\bar{x}\| \right) \le 2 \alpha \|\xdag-\bar{x}\|+\delta, $$
which yields $\underline{\alpha}=\frac{(\tau-1)\delta^\kappa}{2\|\xdag-\bar{x}\|}$.

In the following we choose $\alpha_*=\alpha_*(\delta,\ydel)$ along a fixed geometrically decaying sequence $\alpha_k=q^k\alpha_0$ for some $q\in(0,1)$ according to
\be{discrepancy}
\alpha_*=\max\{\alpha_k\geq0 : \norm{F(x_{\alpha_k}^\delta)-\ydel}\leq \tau \delta^\kappa\}.
\ee
This is a variant of the sequential discrepancy principle, see also \cite{AnzHofMath13}. Well-definedness and strict positivity of $\alpha_*$ follow from the existence of positive values $\alpha_0$ and $\underline{\alpha}$.
 Estimate \req{estFx} together with \req{discrepancy} yields, for $\delta<1$,
\[
\tau\delta^\kappa<\norm{F(x_{\frac{\alpha_*}{q}}^\delta)-\ydel}\leq \frac{\alpha_*}{q}\norm{\xdag-\bar{x}}+\delta \leq \frac{\alpha_*}{q}\norm{\xdag-\bar{x}}+\delta^\kappa,
\]
and  thus, the following lower bound for $\alpha_*$ 
\[
\alpha_*\geq \frac{q(\tau-1)}{\norm{\xdag-\bar{x}}}\delta^\kappa.
\]
Consequently, for sufficiently small $\delta>0$ one has 
$$\frac{\delta}{\alpha_*} \le \delta^{1-\kappa} \frac{\|\xdag-\bar{x}\|}{q(\tau-1)} \to 0  \quad \mbox{as} \quad \delta \to 0, $$
implying \req{aposteriori}. To see \req{limity} from \req{discrepancy}, we use the triangle inequality as
\[
\norm{F(x_{\alpha_*}^\delta)-y}\leq \norm{F(x_{\alpha_*}^\delta)-\ydel}+\delta \leq (\tau+1)\delta^{\kappa}\to 0 \quad  \mbox{ as } \quad \delta\to0\,.
\]
Hence, Proposition~\ref{pro:apo} applies for this specific variety of a discrepancy principle and arbitrary choices of the exponent $0<\kappa<1$. Note that one can show only weak convergence  $x_{\alpha_*(\delta_{n})}^{\delta_{n}} \rightharpoonup \xdag_{min}$ as $n \to \infty$  if the exponent is $\kappa=1$.

{As an alternative a posteriori choice of the regularization parameter for Lavrentiev's method we can also exploit the Lepski\u\i\ or balancing principle (cf., e.g.,~\cite{Mat06}), which even yields convergence rates (see Theorem~\ref{th:rates_vi}).
}

\section{Convergence rates under variational source conditions} \label{sec:rates_var}

As mentioned in the introduction, we investigate below error estimations under several types of source conditions on solutions of \eqref{Fxy}. Actually,
the solution $\xdag$  referred to in all these source conditions has to be  the $\bar{x}$-minimum-norm solution in the current setting of globally monotone operators $F$ defined on Hilbert spaces. However, we prefer the $\xdag$ notation in order to anticipate the more general situations when the solution set $L$ of \eqref{Fxy} is not necessarilty convex and closed.

\subsection{Variational source condition from Tikhonov regularization} \label{subsec:rates_varTikh}

First we consider  the variational source condition
\be{vsc}
\langle \xdag-\bar x,\xdag-x\rangle\leq \beta\norm{\xdag-x}^2+\psi(\|F(x)-F(\xdag)\|) \qquad \mbox{for all} \quad x\in \mathcal{M},
\ee
which up to some factor 2 on the left-hand side coincides (for $\psi(t)=t$) with the variational source condition originally developed in \cite{HKPS07} for obtaining convergence rates in Tikhonov regularization. Here we use a constant  $\beta\in [0,1)$, some index function $\psi$ and a solution $\xdag$ of \eqref{Fxy}.
The set $\mathcal{M} \subset X$ in \req{vsc} must contain all regularized solutions $\xad$ of interest for $\delta>0$ sufficiently small and $\alpha>0$ chosen appropriately.
For instance, the results in Section~\ref{sec:conv} guarantee that,
under condition \eqref{Fweakseqclosed} and
with an appropriate choice of $\alpha$ in dependence of $\delta$ and a fixed but possibly small radius $\rho>0$,
the set $\mathcal{M}=\mathcal{B}_\rho(\xdag)$ contains all $\xad$ for $\delta>0$ sufficiently small.
By homogeneity argument in the case of a linear operator $F=A \in \mathcal{L}(X,X)$, we conclude that \eqref{vsc} is realistic only for index functions decaying  to zero not faster than $\psi(t)\sim t$ as $t \to +0$. Precisely, taking $x:=\xdag+\varepsilon(\bar{x}-\xdag)$ with $\varepsilon>0$ in \eqref{vsc} implies
\[
\varepsilon\norm{\xdag-\bar x}^2\leq
\beta\varepsilon^2\norm{\xdag-\bar{x}}^2+\psi(\varepsilon\norm{A(\xdag-\bar x)}),
\]
i.e., unless the trivial case $\xdag-\bar x\in\mathcal{N}(A)$ holds, we can divide by
$t:=\varepsilon\norm{A(\xdag-\bar x)}$ to conclude that
\be{psi_opt_vsc}
\lim_{t\to +0}\frac{\psi(t)}{t}\geq \underline{c}>0
\ee
for $\underline{c}=\frac{\norm{\xdag-\bar x}^2}{\norm{A(\xdag-\bar x)}}>0$.

As a matter of fact, if $0 \le \beta\leq\frac12$ and $\mathcal{M}$ contains an $\bar x$-minimum-norm solution, then all solutions $\xdag$ satisfying \eqref{vsc} have to be $\bar x$- minimum-norm solutions. Namely, using \eqref{vsc} with $x=\xdag_{min}$ implies
\[
-\frac12\norm{\xdag_{min}-\bar{x}}^2+\frac12\norm{\xdag-\bar{x}}^2+\frac12\norm{\xdag-\xdag_{min}}^2
=\langle \xdag-\bar x,\xdag-\xdag_{min}\rangle\leq \beta\norm{\xdag-\xdag_{min}}^2\,,
\]
i.e., with $0\le \beta\leq\frac12$, we get $\norm{\xdag-\bar{x}}^2\leq \norm{\xdag_{min}-\bar{x}}^2$.


\begin{theorem}\label{th:apriori}
Let $F:X\to X$ be a hemicontinuous and monotone operator and $\mathcal{M}$ be a subset of the Hilbert space $X$ such that \eqref{vsc} is satisfied. Moreover, let $\alpha$ be chosen such that all
 regularized solutions $\xad$ from \req{Lavrentiev} belong to $\mathcal{M}$ for sufficiently small $\delta>0$.
Then one has the estimate
\be{estim1}
\|\xad-\xdag\|^2\leq \frac{{1}}{(1-\beta)^2} \frac{\delta^2}{\alpha^2}+\frac{2}{1-\beta} \psi(\delta+c\alpha),
\ee
with $c=\norm{\xdag-\bar{x}}$ for  $\delta>0$ as above. For the a priori choice
\be{alpha_vsc_Tikh}
\alpha(\delta)\sim\Phi^{-1}(\delta^2)
\ee
with the index function $\Phi(\alpha):=\alpha^2\psi(\alpha)$ which satisfies \eqref{apriori},
this yields the convergence rate
\be{rate_vsc_Tikh}
\|x_{\alpha(\delta)}^\delta-\xdag\|=O\left( \frac{\delta}{\Phi^{-1}(\delta^2)}\right)
=O\left( \sqrt{\psi(\Phi^{-1}(\delta^2))}\right).
\ee
The best possible rate 
{
under condition \req{vsc}
}
occurs for $\psi(t) \sim t$ as
\be{estimate11}
\|x_{\alpha(\delta)}^\delta-\xdag\|=O(\delta^\frac{1}{3}) \qquad \mbox{if} \quad \alpha(\delta)\sim \delta^\frac{2}{3}.
\ee
\end{theorem}

\begin{proof} Let $\gamma:=\langle F(\xad)-F(x^\dagger),\xad-\xdag\rangle$. Note that $\xad\in\mathcal{M}$ for sufficiently small $\delta>0$,  whenever  the a priori parameter choice $\alpha=\alpha(\delta)$ satisfies  condition \req{apriori} guaranteeing (due to Section~\ref{sec:conv}) convergence of $\xad$ to $\xdag$ and if $\mathcal{M}=\mathcal{B}_\rho(\xdag)$ for some $\rho>0$.

From \eqref{bas4}, \eqref{vsc} and \eqref{delta}  it follows that
\begin{eqnarray}
\gamma+\alpha \norm{\xad-\xdag}^2&=& \langle \ydel-y,\xad-\xdag\rangle+\alpha \langle \xdag-\bar x,\xdag-\xad\rangle \label{id1}\\
 &\leq& \delta\|\xad-\xdag\|+\alpha(\beta\norm{\xad-\xdag}^2+\psi(\|F(\xad)-F(\xdag)\|)).
\nonumber
\end{eqnarray}
The monotonicity \req{Fmon} of $F$ implies $\gamma\geq 0$ and thus we can neglect $\gamma$ and estimate further as
\be{eq1}
(1-\beta)\norm{\xad-\xdag}^2\leq \frac{\delta}{\alpha}\|\xad-\xdag\|+\psi(\|F(\xad)-F(\xdag)\|).
\ee
By monotonicity of $\psi$ and \eqref{estFx} this yields
\[
\norm{\xad-\xdag}^2\leq \frac{1}{1-\beta}\frac{\delta}{\alpha}\|\xad-\xdag\|+\frac{1}{1-\beta}\psi(\delta+c\alpha).
\]
Since for any $d,c_1,c_2\geq0$
\be{quadest}
d^2\leq 2c_1d+c_2 \ \Rightarrow \ d^2\leq 4c_1^2+2c_2
\ee
holds, one obtains \eqref{estim1}.
Since  $\Phi$ is an index function, the same is valid  also for  $\Phi^{-1}$, as well as for  the compositions $\Phi^{-1}(\delta^2)$ and $\sqrt{\psi(\Phi^{-1}(\delta^2))}=\frac{\delta}{\Phi^{-1}(\delta^2)}$. Consequently, the a priori regularization parameter choice \req{alpha_vsc_Tikh} satisfies $\delta=O(\alpha(\delta))$ and moreover the limit conditions \eqref{apriori}, which are relevant for convergence (cf.~Section~\ref{sec:conv}). Then the convergence rate
\eqref{rate_vsc_Tikh} can be derived from \eqref{estim1} for the choice \eqref{alpha_vsc_Tikh}.
Due to \eqref{psi_opt_vsc}, we have  the inequality $\alpha^3 \le c_1 \Phi(\alpha)$, hence $\Phi^{-1}(\delta^2) \le c_2 \delta^{2/3}$ and
$\frac{\delta}{\Phi^{-1}(\delta^2)} \ge c_3 \delta^{1/3}$, for some positive constants $c_1,c_2$ and $c_3$. This shows that $O(\delta^{1/3})$ is the best possible rate attained for $\psi(t)\sim t$.
\end{proof}

\begin{remark}  {\rm The classical `adjoint' source condition $\xdag-\bar x \in \mathcal{R}(A^*)$ (which implies \eqref{vsc} with $\beta=0$ and $\psi(t)=t$ in the linear operator case) yields the convergence rate $O(\delta^{1/3})$. 
{Thus, the rate $O(\delta^{1/2})$ mentioned with no proof by Theorem 6 in \cite{LiuNash96} seems to be a typo.}
    However, the `no adjoint' source condition $\xdag-\bar x \in \mathcal{R}(A)$ occurring in \eqref{LavHoelder} for  $p=1$ implies the better rate $O(\delta^{1/2})$ (cf.~\cite{Taut02}). This is due to the higher rate $O(\alpha)$ for the approximation error (the exact data case), as compared to $O(\alpha^{1/2})$ derived from $\xdag-\bar x \in \mathcal{R}(A^*)$ (see, e.g., Theorem 5 in \cite{LiuNash96}), while the propagated data error is the same in both situations: $O(\delta/\alpha)$.
}\end{remark}\bigskip

\subsection{Variational source condition specific to Lavrentiev regularization} \label{subsec:rates_varLavr}

In this section we propose a variational source condition which seems to naturally fit the setting of Lavrentiev's method.

Recently, results in \cite{Sch_etal15} regarding Tikhonov regularization in the linear case point out connections between the standard source condition $x^\dagger\in\mathcal{R}((A^*A)^{\nu/2})$ for $\nu\in(0,1]$ and several types of variational source conditions along with the one
\be{vscHilb}
2\langle x^\dagger,x\rangle\leq\beta_1\|x\|^2+\beta_2\|Ax\|^\frac{2\nu}{\nu+1} \qquad \mbox{for all} \quad x\in X,
\ee
with constants $\beta_1\in[0,1)$ and $\beta_2\geq 0$, which is very similar to \eqref{vsc} for {power-type functions} $\psi$.
Searching for an adaption to the Lavrentiev setting,  it makes sense however to consider \eqref{lvsc} instead of \eqref{vscHilb},
\be{lvsc}
\langle x^\dagger,x\rangle\leq\beta_1\|x\|^2+\beta_2\langle Ax,x\rangle^\mu \qquad \mbox{for all} \quad x\in X,
\ee
with $\mu>0$ and constants $\beta_1,\beta_2$ as above.
Actually, we will deal with the more general version with respect to the nonlinear case
\be{lvsc_gen}
\langle x^\dagger-\bar x,\xdag-x\rangle\leq\beta\|x-\xdag\|^2
+\varphi(\langle F(x)-F(\xdag),x-\xdag\rangle) \qquad \mbox{for all} \quad x\in \mathcal{M},
\ee
with $\beta\in[0,1)$ and $\varphi$ an index function.
Again, no condition additional to monotonicity needs to be imposed on $F$.
As in Section~\ref{subsec:rates_varTikh}, there is a natural upper bound on the decay rate of $\varphi$ as $t$ tends to zero. To this end, consider again the linear case $F=A \in \mathcal{L}(X,X) $ and $x:=\xdag+\varepsilon(\bar{x}-\xdag)$ with $\varepsilon>0$ in \eqref{lvsc_gen}. This implies
\[
\varepsilon\norm{\xdag-\bar x}^2\leq
\beta\varepsilon^2\norm{\xdag-\bar{x}}^2
+\varphi(\varepsilon^2\langle A(\xdag-\bar x),\xdag-\bar x\rangle),
\]
which by division to $\sqrt{t}=\varepsilon\sqrt{\langle A(\xdag-\bar x),\xdag-\bar x\rangle}$ and by letting $\varepsilon\to 0$ enforces
\be{phi_opt_vsc}
\lim_{t\to0}\frac{\varphi(t)}{\sqrt{t}}\geq \underline{c}>0
\ee
with the constant $\underline{c}=\frac{\norm{\xdag-\bar{x}}}{\sqrt{\langle A(\xdag-\bar x),\xdag-\bar x\rangle}}>0$, unless the trivial case \linebreak $\langle A(\xdag-\bar x),\xdag-\bar x\rangle=0$ holds,
so that the best rate obtainable here is the one associated with  $\varphi(t)=\sqrt{t}$.

{
To quantify convergence rates, we introduce the index function $\Psi$ as follows.
Let $f$ be an index function such that its antiderivative $\tilde f(s):=\int_0^s f(t) dt$ satisfies
the condition $\tilde f(\varphi(s)) \le s$ for $s>0$ and let $G(\alpha) \ge \int_0^\alpha f^{-1}(\tau) d\tau$. Then we set $\Psi(\alpha):=\frac{G(\alpha)}{\alpha}$.
\\
We will first of all obtain rates by choosing
\be{alpha_vsc_Lavr}
\alpha_{apri}=\alpha(\delta):=\Theta^{-1}(\delta^2)
\ee
with
\begin{equation}\label{Theta}
\Theta(\alpha):=\alpha^2 \Psi(\alpha),
\end{equation}
which also satisfies  condition \eqref{apriori} for convergence without source condition.
\\
The same rates are achieved with an a posteriori parameter choice $\alpha_{Lep}$ according to the Lepski\u\i\ principle, which does not require the knowledge of $\Psi$. In this context, the regularization parameter $\alpha_{Lep}$ is determined as follows:
For prescribed $\tau >1$ and  $0 < q < 1$ we consider the function
$$ \Sigma(\alpha,\delta):= \left(\frac{\sqrt{3-2\beta}}{1-\beta}\right) \,\frac{\delta}{\alpha}$$
and the increasing geometric sequence
$$ \Delta_q := \{ \alpha_j : \;\alpha_j=\alpha_0/q^j, \; j \in \mathbb{N} \}$$
with $\alpha_0>0$ sufficiently small, such that $\|x^\delta_{\alpha_0}-\xdag\| \le \Sigma(\alpha_0,\delta)$. Then we set
\begin{equation}\label{alpha_vsc_Lavr_Leps}
\begin{aligned}
&\alpha_{Lep}=\alpha(\delta,y^\delta):=\\
&\max\{\alpha \in \Delta_q: \,\|x_{\alpha^\prime}^\delta-x_{\alpha}^\delta\| \le 2\, \Sigma(\alpha^\prime,\delta) \mbox{ for all } \alpha^\prime\in\Delta_q\cap [\alpha_0, \alpha]\}.
\end{aligned}
\end{equation}
\begin{theorem}\label{th:rates_vi}
Let $F:X\to X$ be a hemicontinuous and monotone operator and $\mathcal{M}$ be a subset of the Hilbert space $X$ such that \eqref{lvsc_gen} is satisfied. Moreover, let $\alpha$ be chosen such that all
 regularized solutions $\xad$ from \req{Lavrentiev} belong to $\mathcal{M}$ for sufficiently small $\delta>0$.
Then one has the estimate
\be{estim1_Lavr}
\|\xad-\xdag\|^2\leq \frac{{1}}{(1-\beta)^2} \frac{\delta^2}{\alpha^2}+\frac{2}{1-\beta} \Psi(\alpha),
\ee
for  $\delta>0$ as above.
For the a priori choice \eqref{alpha_vsc_Lavr}, this yields the convergence rate
\be{rate_vsc_Lavr}
\|x_{\alpha_{apri}}^\delta-\xdag\|=O\left( \frac{\delta}{\Theta^{-1}(\delta^2)}\right)
=O\left( \sqrt{\Psi(\Theta^{-1}(\delta^2))}\right),
\ee
with $\Theta$ as in \eqref{Theta}.
The best possible rate occurs for $\varphi(t)\sim\sqrt{t}$ and $\Psi(t)\sim t$ as
\be{estimate11_Lavr}
\|x_{\alpha_{apri}}^\delta-\xdag\|=O(\delta^\frac{1}{3}) \qquad
\mbox{if} \quad \alpha(\delta)\sim \delta^\frac{2}{3}.
\ee
The convergence rate \eqref{rate_vsc_Lavr} is also obtained with the Lepski\u\i\ principle \eqref{alpha_vsc_Lavr_Leps}
$$
\|x_{\alpha_{Lep}}^\delta-\xdag\|=O\left( \frac{\delta}{\Theta^{-1}(\delta^2)}\right)\,.
$$
\end{theorem}
}

\begin{proof} Note that $\xad\in\mathcal{M}$, for all $\delta>0$ sufficiently small, due to the convergence results in the a priori parameter choice case.
Let again $\gamma:=\langle F(\xad)-F(x^\dagger),\xad-\xdag\rangle\geq 0$.
By \eqref{id1} and \eqref{lvsc_gen} written for $x:=\xad$ one has
\[
\gamma+\alpha \|\xad-\xdag\|^2
\leq \delta\|\xad-\xdag\|+\alpha(\beta\|\xad-\xdag\|^2+\varphi(\gamma)).
\]
Now one can employ Young's inequality
$$ab\leq \int_0^a f(t)\,dt+\int_0^bf^{-1}(t)\,dt$$  for an index function $f$ which will be specified later, for
$a=\varphi(\gamma)$ and $b=\alpha$. Using the settings
$$\tilde{f}(s)=\int_0^s f(t)\,dt,\,\,\,\,\,G(\alpha) \ge \int_0^\alpha f^{-1}(t)\,dt,$$
one needs to find the function $f$ such that
$\tilde{f}(\varphi(\gamma))\leq\gamma$.
Then one has
$$ \alpha(1-\beta) \|\xad-\xdag\|^2\leq\delta\|\xad-\xdag\|+G(\alpha),$$
so with $\Psi(\alpha)=\frac{G(\alpha)}{\alpha}$ one gets
$$
\|\xad-\xdag\|^2\leq\frac{1}{1-\beta}\,\left(\frac{\delta}{\alpha}\|\xad-\xdag\|+\Psi(\alpha)\right).
$$
We use again the implication \eqref{quadest} to conclude \eqref{estim1_Lavr}.
{The a priori parameter choice \eqref{alpha_vsc_Lavr} yields the convergence rate \eqref{rate_vsc_Lavr}. For showing the same rate for the a posteriori parameter choice according to 
the Lepski\u\i\ principle, we apply Proposition~1 from \cite{Mat06} (see also \cite[\S~4.2.2]{HofMat12}). Taking into account that $\Psi$ is an increasing function, we have  from  \eqref{estim1_Lavr} that
$$\|\xad-\xdag\|^2 \le \frac{\delta^2+2(1-\beta)\alpha_{apri}^2\Psi(\alpha_{apri})}{(1-\beta)^2\,\alpha^2}= \frac{\delta^2+2(1-\beta)\delta^2}{(1-\beta)^2\,\alpha^2}=(\Sigma(\alpha,\delta))^2,$$
for all $\alpha_0 \le \alpha \le \alpha_{apri}$.
This  ensures the estimate $\|x_{\alpha_{Lep}}^\delta-\xdag\| \le 3\, \Sigma(\alpha_{apri},\delta)=O\left( \frac{\delta}{\Theta^{-1}(\delta^2)}\right)$ (cf. \cite[Proposition~1]{Mat06}).
}\end{proof}

\begin{remark}
{\rm The construction of $f,G$ and $\Psi$ in Theorem~\ref{th:rates_vi} is straightforward if $\varphi$ is smooth enough and we can use $f(t) :=(\varphi^{-1})^\prime(t)$ and $G(\alpha):=\int_0^\alpha f^{-1}(\tau)d\tau$.
Obviously, $\Psi$ in Theorem~\ref{th:rates_vi} plays a similar role as $\psi$ did in Theorem~\ref{th:apriori}.
}\end{remark}

\subsection{The special case of H\"older and logarithmic source conditions}

\paragraph{H\"older-type variational source conditions.} \label{subHoe}

Taking into account exponents $\mu\in(0,\frac12]$, we consider \eqref{lvsc_gen} with
$$\varphi(t)=t^\mu,\,\,\,\,\,f(t)=\frac{1}{\mu}t^\frac{1-\mu}{\mu},\,\,\,\,\,f^{-1}(t)=(\mu t)^\frac{\mu}{1-\mu},\,\,\,\tilde{f}(s)=s^{\frac{1}{\mu}},$$
$$G(\alpha)=(1-\mu)\mu^{\frac{\mu}{1-\mu}}\alpha^\frac{1}{1-\mu},\,\,\,\,\,\Psi(\alpha)=(1-\mu)\mu^{\frac{\mu}{1-\mu}}\alpha^\frac{\mu}{1-\mu},\,\,\,\,\,\Theta(t)\sim t^\frac{2-\mu}{1-\mu},$$
or correspondingly \eqref{vsc} with
$$\psi(t)\sim t^{\frac{\mu}{1-\mu}}.$$
According to Theorems~\ref{th:apriori} and \ref{th:rates_vi} this yields in both cases the H\"older rate
\[
\|x_{\alpha(\delta)}^\delta-\xdag\|=O(\delta^\frac{\mu}{2-\mu}) \qquad
\mbox{if} \quad \alpha(\delta)\sim \delta^\frac{2(1-\mu)}{2-\mu}.
\]

\paragraph{Logarithmic-type variational source conditions.}

We consider \eqref{lvsc_gen} with
$$\varphi(t)=\frac{1}{-\ln t},\,\,\,\,\,f(t)=e^{-\frac{1}{t}}\leq \frac{1}{t^2}e^{-\frac{1}{t}}={\varphi^{-1}}'(t),\,\,\,\,\,f^{-1}(t)=\frac{1}{-\ln t},$$
$$G(\alpha)=\alpha\frac{1}{-\ln \alpha}\geq \int_0^\alpha \frac{1}{-\ln t}\,dt
,\,\,\,\,\,\Theta(t)\sim t^2 \frac{1}{-\ln t},$$
for $\alpha,t\in(0,1)$ or correspondingly \eqref{vsc} with
$$\psi(t)\sim \frac{1}{-\ln t}.$$
According to Theorems~\ref{th:apriori} and \ref{th:rates_vi}, this yields again in both cases the same logarithmic rate $O\left( \sqrt{\Psi(\Theta^{-1}(\delta^2))}\right)$.

Since the a priori choice $\alpha(\delta)\sim\Theta^{-1}(\delta^2)$ can not be determined explicitly in this logarithmic case, a more convenient choice  is $\alpha(\delta)\sim \sqrt{\delta}$ which implies
\[
\|x_{\alpha(\delta)}^\delta-\xdag\|=O\left(\frac{1}{\sqrt{-\ln(\delta)}}\right) \qquad.
\]
For other results concerning logarithmic rates based on variational source conditions we also refer, for example, to \cite{BoHo10,HohWei15,WeHo12}.

%
%
%
%
%
%
%

\begin{remark}
At the moment we cannot present convergence rates for the Lavrentiev regularization under variational source conditions with the variant of a discrepancy principle, for which the convergence was shown in Section~\ref{sec:conv}. Perhaps this is due to an intrinsic deficit of discrepancy principles in Lavrentiev regularization. Thus, {we have considered the balancing principle \req{alpha_vsc_Lavr_Leps} as an alternative a posteriori rule for proving rates, see Theorem \ref{th:rates_vi}.
}
  \end{remark}

\section{Convergence rates under approximate source conditions} \label{sec:rates_approx}
We assume that $F$ locally satisfies an approximate range invariance condition, which means that there exist a bounded monotone (accretive) linear operator $A \in \mathcal{L}(X,X)$ satisfying the condition
\begin{equation} \label{accretive}
\langle Ax,x\rangle \geq 0  \qquad \mbox{for all} \quad x\in X
\end{equation}
and a ball $\mathcal{B}_\rho(\xdag) \subseteq \mathcal{D}(F)$ of positive, but maybe small, radius $\rho>0$ and a uniform constant $0<c<1$, such that for all points $x \in \mathcal{B}_\rho(\xdag)$ there are bounded linear operators $M(x),N(x)\in \mathcal{L}(X,X)$
 satisfying the conditions
\begin{equation} \label{invariance}
\begin{aligned}
&F(x)-F(\xdag)=A M(x)(x-\xdag)+N(x) \quad  \mbox{ and }\\
&\norm{M(x)-I}\leq c<1\,, \quad \norm{N(x)} \leq \sigma(\norm{x-\xdag})
\end{aligned}
\end{equation}
with $\sigma$ to be specified below (cf. \req{sigma}).
If $N$ does not vanish identically, we additionally assume that
\begin{equation}\label{smallness}
\begin{aligned}
&\forall x,\tilde{x}\in \mathcal{B}_\rho(\xdag),\, \norm{M(x)-M(\tilde{x})}\leq L_M\norm{x-\tilde{x}}\mbox{ and } AM(x) \mbox{ accretive }\\
&\norm{\xdag-\bar{x}}\leq (1-c)\rho \,, \quad c+\rho L_M<1
\end{aligned}
\end{equation}
holds.
Note that by the convergence result from Section~\ref{sec:conv}, there exists $\bar{\delta}>0$ such that for all $\delta\in(0,\bar{\delta}]$ the regularized solution $x_{\alpha_*(\delta)}^\delta$ lies in $\mathcal{B}_\rho(\xdag)$.

Accretivity \eqref{accretive} of $A$ has the consequence that (see, e.g., Proposition 2.1 in \cite{Taut02})
\begin{equation} \label{termconsequence}
 \norm{(A+\alpha I)^{-1}}\leq\frac{1}{\alpha} \quad \mbox{and} \quad \norm{(A+\alpha I)^{-1}A}\leq 1 \qquad \mbox{for all} \quad \alpha>0.
\end{equation}
If the monotone operator $F$ is differentiable at $\hat x \in \mathcal{D}(F)$ with Fr\'echet derivative $F^\prime(\hat x)$, then it is well-known that $A:=F^\prime(\hat x)$ is accretive and thus satisfies  conditions \eqref{termconsequence}.

\subsection{General convergence rates}

The benchmark source condition for an accretive linear operator $A \in \mathcal{L}(X,X)$ we consider here is
\begin{equation} \label{exact0}
\xdag-\bar{x}=Aw, \quad \mbox{for some} \;\; w \in X.
\end{equation}
Moreover, we can use for arbitrary $R>0$ an approximate source condition of the form
\begin{equation} \label{apprsc}
\xdag-\bar{x}=Aw_R+r_R\qquad w_R,r_R \in X, \; \|w_R\| = R,\;\|r_R\|=d(R),
\end{equation}
with a distance function
\begin{equation} \label{distfunct}
d(R)=\min\{\norm{\xdag-\bar{x}-Aw} \ : \ \|w\|\leq R\}.
\end{equation}
If the benchmark source condition (\ref{exact0}) fails and hence $\|r_R\|>0$ for all $R>0$ in (\ref{apprsc}),
then the method of approximate source conditions using the distance function \eqref{distfunct} applies.
This method was developed for linear ill-posed operator equations in \cite{Hof06} (see also \cite{DHY07,HDK06} and ideas in \cite{BakuKok04}).
An extension to nonlinear operator equation can be found in \cite{HeinHof09} and moreover in \cite{BoHo10}.
It yields convergence rates
whenever $d(R),\;R>0,$ is continuous, positive, convex and strictly decreasing to zero as $R \to \infty$.
This is the case if
\begin{equation} \label{Astar}
\overline{\mathcal{R}(A)}^X=X.
\end{equation}
(see, e.g.,~\cite[Lemma~3.2]{BoHo10}).
For accretive $A$ this
is equivalent to the injectivity of $A$ due to $\overline{\mathcal{R}(A)}\oplus \mathcal{N}(A)=X$
(cf.~\cite[Prop.~2.1.1 h]{Haase06}).

Under the regularity conditions \req{exact0}, \req{apprsc}, we will use
the a priori parameter choice $\alpha\sim \phi^{-1}(\delta)$ in the sense that
there exists a sufficiently large constant $\tau>1$ such that
\begin{equation}\label{alphadelta}
\begin{cases}
\delta\leq \tau \phi(\alpha) \,,\mbox{ and }
d(\chi^{-1}(\alpha))\leq \tau d(\chi^{-1}(\phi^{-1}(\delta))
\mbox{ if \req{exact0} is violated}\\
\frac{1}{\tau}\delta\leq \alpha^2\leq \tau\delta
\mbox{ if \req{exact0} is satisfied}
\end{cases}
\end{equation}
and the assumption
\begin{equation}\label{sigma}
\sigma(e)\leq C_\sigma e \varepsilon^{-1}(e)
\end{equation}
holds for some constant $C_\sigma>0$ in \req{invariance},
with the auxiliary functions
\begin{equation}\label{phi}
\phi(\alpha)=\alpha d(\chi^{-1}(\alpha))\,,
\end{equation}
\begin{equation}\label{chi}
\chi(R)=\frac{d(R)}{R}\,,
\end{equation}
\begin{equation}\label{varepsilon}
\varepsilon(\alpha)=
\begin{cases}
\left(\frac{2-c}{1-c}+\tau)d(\chi^{-1}(\alpha)\right)
\mbox{ if \req{exact0} is violated}\\
(\frac{\norm{w}}{1-c}+\tau)\alpha
\mbox{ if \req{exact0} is satisfied.}
\end{cases}
\end{equation}
Note that $d$ is monotonically decreasing by definition and $\chi$ is strictly decreasing; if the benchmark source condition \req{exact0} is violated, then $d$ is strictly decreasing and hence $\varepsilon$ is invertible and $\sigma$ is strictly monotonically increasing.
Therewith, we arrive at the following rates result.
\begin{theorem}\label{th:rates}
Let $F:X\to X$ be a hemicontinuous and monotone operator satisfying \req{invariance} with \req{accretive}, \eqref{Astar} and \req{sigma}, and, if $N\not\equiv0$, additionally \req{smallness}. Then with the choice \req{alphadelta} we have
\[
\norm{\xad-\xdag}\,=\,\begin{cases}\;

O(d(\chi^{-1}(\phi^{-1}(\delta))))
\mbox{ if \req{exact0} is violated},\\\;
O(\sqrt{\delta})
\mbox{ if \req{exact0} is satisfied.}
\end{cases}
\]
In particular, if $F$ is G\^{a}teaux differentiable with locally Lipschitz continuous derivative
\begin{equation} \label{Lipschitz}
 \|F^\prime(x)-F^\prime(\xdag)\| \le L \,\|x-\xdag\|  \qquad \mbox{for all} \quad x \in \mathcal{B}_\rho(\xdag)
\end{equation}
and \req{exact0} holds with $A=F^\prime(\xdag)$, then $\norm{\xad-\xdag}=
O(\sqrt{\delta})$.
\end{theorem}
\begin{proof}
We first consider the case $N\equiv0$ in \req{invariance}, where the estimates are somewhat simpler.
We rewrite \eqref{Lavrentiev} as
\[\begin{aligned}
&\xad-\xdag\\
=&-(A+\alpha I)^{-1}\left(y-\ydel+F(\xad)-F(\xdag)-A(\xad-\xdag)+\alpha(\xdag-\bar{x})\right)\\
=&-(A+\alpha I)^{-1}(y-\ydel)-(A+\alpha I)^{-1}A(M(\xad)-I)(\xad-\xdag)\\
&-\alpha(A+\alpha I)^{-1}Aw_R+\alpha(A+\alpha I)^{-1}r_R\,,
\end{aligned}
\]
with $w_R=w$ and $r_R=0$ in case of \req{exact0}.
Taking into account the assumption \req{invariance} (with $N\equiv0$) and the inequalities (\ref{termconsequence}), we obtain
the estimate
\[
\norm{\xad-\xdag}\leq \frac{\delta}{\alpha}+c\norm{\xad-\xdag}+ \alpha R+ d(R),
\]
i.e.,
\begin{equation}\label{estxadxda}
\norm{\xad-\xdag}\leq \frac{1}{1-c}
\Bigl(\frac{\delta}{\alpha}+ \alpha R+ d(R)\Bigr),
\end{equation}
Considering the case that \req{exact0} is violated and equilibrating the terms containing $R$ we obtain {the best} $R$ in dependence of $\alpha$:
\begin{equation}\label{Ralpha}
R=
\chi^{-1}(\alpha)\,,
\end{equation}
where $\chi$ is defined as in \req{chi}.
Equilibrating now the two terms in the right hand side of the resulting estimate
\[
\norm{\xad-\xdag}\leq \frac{1}{1-c}\left(\frac{\delta}{\alpha}+2d(\chi^{-1}(\alpha))\right)
\]
yields the {best} a priori choice \req{alphadelta}
(where $\phi$ is defined as in \req{phi})
and thus, the rate estimate
\begin{equation}\label{genestimate}
\norm{\xad-\xdag}=O(d(\chi^{-1}(\phi^{-1}(\delta)))) \qquad \mbox{as} \qquad \delta \to 0\,.
\end{equation}
If \req{exact0} holds, the estimate simplifies to
\[
\norm{\xad-\xdag}
\leq \frac{1}{1-c}\Bigl(\frac{\delta}{\alpha}+\alpha\norm{w}\Bigr)
\leq\frac{1}{1-c}\Bigl(\tau+\norm{w} \Bigr)\alpha
\leq \frac{1}{1-c}\Bigl(\tau+\norm{w} \Bigr)\sqrt{\tau\delta}
\]
where we have used \req{alphadelta}.

If $N$ does not vanish identically, we (as in \cite{Taut02}) introduce an intermediate quantity $z$ solving the equation
\begin{equation}\label{z}
(AM(z)+\alpha I)(z-\xdag)+\alpha(\xdag-\bar{x})=0
\end{equation}
with $A$, $M$ as in \req{invariance}.
Solvability of \req{z} follows from Banach's Fixed Point Theorem, applied to the reformulation
\begin{equation}
z=\Phi(z)=\xdag+(A+\alpha I)^{-1}\Bigl(A(I-M(z))(z-\xdag)+\alpha(\bar{x}-\xdag)\Bigr)
\end{equation}
of \req{z}.
Indeed, $\Phi$ is a self-mapping on $\overline{\mathcal{B}_{\rho}(\xdag)}$, since for any $z\in\overline{\mathcal{B}_{\rho}(\xdag)}$
\[
\begin{aligned}
\norm{\Phi(z)-\xdag}
&=
\norm{(A+\alpha I)^{-1}A(I-M(z))(z-\xdag)+\alpha(A+\alpha I)^{-1}(\bar{x}-\xdag)}\\
&\leq
\norm{(I-M(z)(z-\xdag)}+\norm{\bar{x}-\xdag}\\
&\leq c\rho+\norm{\bar{x}-\xdag}\rho \leq \rho,
\end{aligned}
\]
due to \req{termconsequence} and \req{smallness}.
Contractivity of $\Phi$ follows from the estimate
\[
\begin{aligned}
\norm{\Phi(z)-\Phi(\tilde{z})}
&=
\norm{(A+\alpha I)^{-1}A\Bigl((I-M(z))(z-\xdag)-(I-M(\tilde{z}))(\tilde{z}-\xdag)\Bigr)}\\
&\leq
\norm{(I-M(z))(z-\tilde{z})}+\norm{(M(z)-M(\tilde{z}))(\tilde{z}-\xdag)}\\
&\leq
(c+L_M\rho)\norm{z-\tilde{z}}
\end{aligned}
\]
for any $z,\tilde{z}\in \overline{\mathcal{B}_{\rho}(\xdag)}$, where $c+L_M\rho<1$ by \req{smallness}.

Using \req{z}, we can rewrite \eqref{Lavrentiev} as
\[
F(\xad)-F(z)+y-\ydel+F(z)-F(\xdag)-AM(z)(z-\xdag)+\alpha(\xad-z)=0.
\]
By testing it with $\xad-z$, by monotonicity of $F$ and \req{invariance}, we obtain
\begin{equation}\label{estxz}
\norm{\xad-z}\leq \frac{\delta}{\alpha}+\frac{\sigma(\norm{z-\xdag})}{\alpha}\,.
\end{equation}
Here, by employing  \req{termconsequence} and  the approximate source condition \req{apprsc}
with \\
$\norm{M(z)^{-1}w_R}\leq \frac{R}{1-c}$, $\norm{r_R}\leq d(R)$,
we can estimate as follows,
\[
\begin{aligned}
\norm{z-\xdag}
&=\alpha\norm{(AM(z)+\alpha I)^{-1}(\xdag-\bar{x})}\\
&=\alpha\norm{(AM(z)+\alpha I)^{-1}(AM(z)M(z)^{-1}w_R+r_R)}
\leq \frac{R}{1-c}\alpha + d(R),
\end{aligned}
\]
Using this and the triangle inequality in \req{estxz} we get
\[
\norm{\xad-\xdag}\leq \frac{\delta}{\alpha}+\frac{R}{1-c}\alpha + d(R)
+\frac{\sigma(\frac{R}{1-c}\alpha + d(R))}{\alpha}\,,
\]
since $\sigma$ is monotonically increasing.
By  choosing the interdependence of $R$, $\alpha$, and $\delta$ {in the best way} as above, cf. \req{Ralpha}, \req{alphadelta}, we can achieve that
\[
\frac{\delta}{\alpha}+\frac{R}{1-c}\alpha + d(R)
\leq \varepsilon(\alpha)
\]
with $\varepsilon$ according to \req{varepsilon}.
So in order to maintain the rate from the case $N\equiv0$, we assume the  function $\sigma:\R^+\to\R^+$ to be monotonically increasing and satisfying
\req{sigma}.
This gives the estimate
\[
\norm{\xad-\xdag}\leq \varepsilon(\alpha)
+\frac{\sigma(\varepsilon(\alpha))}{\alpha}
\leq (1+C_\sigma)\varepsilon(\alpha),
\]
which yields   the claimed rates  due to \req{alphadelta}.
\end{proof}

\begin{remark}
{\rm  If \eqref{exact0} fails, it follows due to  \eqref{Astar} that the function $d(R),\;R>0$, given by (\ref{distfunct}) is continuous, positive, convex and strictly decreasing to zero as $R \to \infty$, the rate function $d(\chi^{-1}(\phi^{-1}(\delta))$ is strictly increasing for $\delta>0$
and satisfies the limit conditions $\lim \limits_{\delta \to +0} d(\chi^{-1}(\phi^{-1}(\delta)))=0$ and $\lim \limits_{\delta \to +0} \frac{\sqrt{\delta}}{d(\chi^{-1}(\phi^{-1}(\delta)))}=0$. The latter condition shows that
we have a lower convergence rate if the exact source condition (\ref{exact}) fails.

This approach also works if $d(R)$ stands for a majorant function to the distance function (\ref{distfunct}) which is strictly decreasing to zero as $R \to \infty$.

If $d$ is differentiable, then instead of equilibrating terms one can as well minimize the right hand side of \req{estxadxda} with respect to $\alpha$ and $R$, which leads to the same result at least in the H\"older case considered in Section \ref{subproof}.
}\end{remark}

\subsection{Comparison to previous results}

If the benchmark source condition
\begin{equation} \label{exact}
\xdag-\bar{x}=Aw, \qquad w \in X,\;\|w\|=R_0>0,
\end{equation}
is valid, then (\ref{apprsc}) holds with $w_{R_0}=w$ and $d(R)=0,\;R_0 \le R<\infty$.
In that case we have the benchmark rate
\begin{equation} \label{benchmark}
\norm{\xad-\xdag}=O(\sqrt{\delta})\qquad \mbox{as} \qquad \delta \to 0,
\end{equation}
whenever $\alpha \sim \sqrt{\delta}$. This was already shown in \cite{Taut02} for the Fr\'echet derivative $F^\prime(\xdag)\in \mathcal{L}(X,X) $ as linear operator $A$,
under the weaker nonlinearity condition \req{Lipschitz} that we also use in Theorem \ref{th:rates}. Note that for accretive $A \in \mathcal{L}(X,X)$, condition \req{exact} requires that $\xdag-\bar{x}$ is orthogonal to the nullspace $\mathcal{N}(A)$.

The range invariance condition (\ref{invariance}) with $N\equiv0$ (cf., e.g., \cite{HMP07,Kal98,ScEnKu93,TaJi03}) is really a strong nonlinearity condition, but it has the advantage that $A$ need not be exactly equal to $F^\prime(\xdag)$, but can be chosen from a wider
variety of linear operators. For example, the nonlinearity condition involving constants $k_0,\rho>0$ such that for all $x,\hat{x} \in \mathcal{B}_{\rho}(\xdag) \subset \mathcal{D}(F)$ and $h \in X$ there
exist elements $k(x,\hat{x},h) \in X$ with the property
\begin{equation} \label{rangeturn}
[F^\prime(x)-F^\prime(\hat{x})]\,h=F^\prime(\hat{x})\,k(x,\hat{x},h) \quad \mbox{and} \quad \|k(x,\hat{x},h)\| \le k_0\|x-\hat{x}\|\,\|h\|,
\end{equation}
is frequently used
(see, e.g., \cite[conditions (10), (11)]{ScEnKu93}).
Due to the mean value theorem, this implies  the range invariance condition (\ref{invariance}) with $A=F^\prime(\xdag)$ provided that the smallness condition  $k_0\,\|F^\prime(\xdag)\|<1$ holds. If the following inequality is satisfied for some constants $0<\kappa \le 1$ and $k_1>0$,
$$\|k(x,\xdag,h)\| \le k_1\,\|x-\xdag\|^\kappa\,\|h\|, $$
then the smallness condition is fulfilled whenever $\rho$ is sufficiently small.

{Condition (\ref{rangeturn}) with fixed $\hat{x}:=\xdag$ occurs as Assumption~A3 in \cite{Taut02}. Note that the operators  $A^p,\;0<p \le 1,$ mentioned at \eqref{LavHoelder} are  fractional powers  of the accretive (sectorial) operator $A$, defined by a Dunford integral as
\begin{equation} \label{powerdef}
A^p v:= \frac{\sin(p\pi)}{\pi}\,\int \limits_0^\infty s^{p-1}(A+sI)^{-1}Av\, ds, \qquad v \in X.
\end{equation} This approach was used there to prove H\"older convergence rates
\begin{equation} \label{Hoelderrate}
\norm{\xad-\xdag}=O\left(\delta^\frac{p}{p+1}\right)\qquad \mbox{as} \qquad \delta \to 0,
\end{equation}
with $A:=F^\prime(\xdag)$ for Lavrentiev regularization and with an a priori parameter choice $\alpha \sim \delta^\frac{1}{p+1}$, under the source condition
\begin{equation} \label{fractional}
\xdag-\bar{x}=A^p\,w, \qquad w \in X,
\end{equation}
(see also~\cite[Chapt.~8]{Pruess93} or \cite[Chapt.~1]{Plato95}).} We will obtain this rate result in Section~\ref{subproof} under the range invariance condition (\ref{invariance}) with a completely different proof.

If (\ref{rangeturn}) holds for all center elements $\hat{x} \in \mathcal{B}_{\rho}(\xdag)$ and all $x \in \mathcal{B}_{\rho}(\xdag)$, then one even has a range invariance condition (\ref{invariance}) with $A=F^\prime(\hat{x})$ for all such $\hat{x}$ whenever
 $$\|k(x,\hat{x},h)\| \le k_1\,\|x-\hat{x}\|^\kappa\,\|h\|,\,\,\,\forall h\in X $$
 and $\rho$ is sufficiently small.

Extension of this type of conditions to \eqref{invariance} with $N\not=0$ allows for situations in which the ranges of the linearizations do not coincide exactly. Exact coincidence of the ranges can be quite restrictive if these ranges are non closed (as relevant in ill-posed problems). In the situation of dense but non-closed range, a small correction $N$ might often be sufficient to bridge the gap between ranges of linearizations at different points.

\subsection{More on H\"older convergence rates} \label{subproof}

{
\begin{proposition} \label{powerrate}
Let $F: X \to X$ be a hemicontinuous and monotone operator and suppose that a range invariance condition (\ref{invariance}) holds for some accretive linear operator $A$ satisfying (\ref{Astar}) (with additionally \req{smallness}  if $N\not\equiv0$).
Provided that the benchmark source condition (\ref{exact}) fails for all $R_0>0$, we have a H\"older convergence rate (\ref{Hoelderrate}) for Lavrentiev regularized solutions for an a priori parameter
choice $\alpha \sim \delta^\frac{1}{p+1}$ if there is some source element $w \in X$ such that the fractional power source condition (\ref{fractional}) with $A^p$ from (\ref{powerdef}) is satisfied with   $0<p<1$.
\end{proposition}}
\begin{proof}
Under the assumptions of the proposition,  estimate (\ref{genestimate}) is applicable and thus we only have to show that, for all $0<p<1$ and under (\ref{fractional}),
\begin{equation}\label{sim}
d(\chi^{-1}(\phi^{-1}(\delta)))=O(\delta^\frac{p}{p+1}) \qquad \mbox{as} \qquad \delta \to 0.
\end{equation}
In this context, we exploit the formulas
\begin{equation} \label{aux}
\|(AA^*+\lambda I)^{-1}A^p\| \le \frac{C_1}{\lambda^{1-\frac{p}{2}}} \qquad \mbox{and} \qquad \|A^*(AA^*+\lambda I)^{-1}A^p\| \le \frac{C_2}{\lambda^{\frac{1-p}{2}}},
\end{equation}
which are valid for all $\lambda>0$ and constants $C_1,C_2>0$ depending on $p \in (0,1)$. These formulas are a consequence of the moment inequality for sectorial operators $A$ (cf.~\cite[Prop.~6.6.4]{Haase06}).

The minimum problem (\ref{distfunct}) for verifying $d(R)$ can be reformulated by the Lagrange multiplier method, where for any $\lambda>0$ the element $v_\lambda \in X$ is the uniquely determined minimizer to $$\|\xdag-\bar{x}-Av\|^2+\lambda \|v\|^2 \to \min, \quad \mbox{subject to} \quad v \in X,$$
which can be written as
$$v_\lambda=(A^*A+\lambda I)^{-1}A^*(\xdag-\bar{x})=A^*(AA^*+\lambda I)^{-1}(\xdag-\bar{x}).$$
Since by (\ref{Astar}) the range of $A$ is dense in $X$ and since the source condition (\ref{exact}) fails, we have that the strictly decreasing function $\theta(\lambda):=\|v_\lambda\|^2=\|A^*(AA^*+\lambda I)^{-1}(\xdag-\bar{x})\|^2$ satisfies  the limit conditions $$\lim \limits_{\lambda \to 0} \theta(\lambda)=+\infty \qquad \mbox{and} \qquad \lim \limits_{\lambda \to +\infty} \theta(\lambda)=0.$$
Then there is a uniquely determined $\lambda_R>0$ such that $\theta(\lambda_R)=R^2$ and $d(R)=\|\xdag-\bar{x}-Av_{\lambda_R}\|$.
Based on (\ref{fractional}) and on the first inequality in (\ref{aux}), we derive the estimate
$$d(R) \le \|\lambda_R(AA^*+\lambda_R I)^{-1}A^pw\|\le \lambda_R \|(AA^*+\lambda_R I)^{-1}A^p\|\|w\|\le C_1 \lambda_R^{\,\frac{p}{2}}\|w\|.  $$
We claim that the following decay rate can be established
\begin{equation} \label{pp}
d(R) \le K\,R^\frac{p}{p-1}, \qquad 0<R<\infty,
\end{equation}
for some constant $K>0$. Then by using the majorant function  $K\,R^\frac{p}{p-1}$ as $d(R)$ in formula (\ref{genestimate}), one obtains
the required condition (\ref{sim}) with $\chi(R)\sim R^\frac{1}{p-1},\;\phi(\alpha)\sim \alpha^{p+1}$ and under the a priori parameter
choice $\alpha \sim \delta^\frac{1}{p+1}$. Thus, it remains to show (\ref{pp}). By the second inequality in (\ref{aux}) we have $\theta(\lambda)\le C_2^2\|w\|^2 \lambda^{p-1}$ and thus the value
$\lambda_{maj}=\left(\frac{R}{C_2\|w\|}\right)^\frac{2}{p-1}>0$ solving the
equation $C_2\|w\| \lambda_{maj}^{\,\,\frac{p-1}{2}}=R$ satisfies the inequality $\lambda_R \le \lambda_{maj}$. This implies the estimate (\ref{pp}) and completes the proof.
\end{proof}

{
\begin{remark} \label{rem:p2}
{\rm Due to the identity $$\mathcal{R}(A^p)=\mathcal{R}((AA^*)^{p/2})$$ (cf.,~e.g.,~\cite[Lemma~1]{PlMaHo16}), which is valid for all $0<p \le 1$,
the source condition (\ref{fractional}) is equivalent to
\begin{equation} \label{adjpower}
\xdag-\bar{x}=(AA^*)^{p/2}\,w,\qquad w \in X,
\end{equation}
and Proposition~\ref{powerrate} remains true if (\ref{fractional}) is replaced with (\ref{adjpower}). In that version the proof can be found as
a direct consequence of Theorem~3.2 in \cite{DHY07} (extendable also to non-compact operators $A$) by using $A$ instead of $A^*$, while taking into account that $A \in \mathcal{L}(X,X)$. This also yields the decay rate (\ref{pp}) and hence the required H\"older rate result.
}\end{remark}}

\begin{remark} \label{rem:scexvar}
{\rm {In order to provide a relation between the variational source conditions from Subsection~\ref{subsec:rates_varLavr} as well as from Subsection~\ref{subsec:rates_varTikh} and H\"older source conditions of range-type,
we consider the special case of a monotone operator $F:=A \in \mathcal{L}(X,X)$ and we set for simplicity  $\bar{x}:=0$.}

{Firstly, let $A=A^*$ be a self-adjoint operator and let the range-type source condition \eqref{fractional}  hold for some $0<p \le 1$,
which implies the convergence rate \eqref{Hoelderrate}.} Note that only for $0<p \le 1/2$ we have the inequality chain
\begin{eqnarray*}
\langle\xdag,x\rangle &=&\langle w, A^p x\rangle
\leq \norm{w}\,\norm{A^p x}\\
&\leq& \norm{w}\, \norm{x}^{1-2p}\norm{A^\frac12x}^{2p}
= \norm{w}\, \norm{x}^{1-2p}\langle Ax,x\rangle^{p}
\end{eqnarray*}
based on the interpolation inequality. By using Young's inequality \linebreak
$ab\leq \frac{a^\xi}{\xi}+\frac{b^\eta}{\eta}$ with $a=\norm{x}^{1-2p},\, b=\norm{w}\,\langle Ax,x\rangle^{p},\, \xi=\frac{2}{1-2p}$ and $\eta=\frac{2}{1+2p}$
this yields (for all $0<p \le \frac{1}{2}$) the variational source condition
\be{remvi}
\langle\xdag,x\rangle  \le \left( \frac{1}{2}-p \right)\|x\|^2+\left(\frac{1}{2}+p \right)\|w\|^\frac{2}{2p+1}\,\langle Ax,x\rangle^{\frac{2p}{2p+1}}
\ee
of type \eqref{lvsc} occurring in Subsection~\ref{subsec:rates_varLavr} with $\mu=\frac{2p}{2p+1}$ and valid for all $x \in X$.
Evidently, from \req{remvi} we derive in the sense of Subsection~\ref{subHoe} the H\"older convergence rate
\be{eq:rate1}
\|x_{\alpha(\delta)}^\delta-\xdag\|=O(\delta^\frac{\mu}{2-\mu})=O(\delta^\frac{p}{p+1}) \qquad
\mbox{if} \quad \alpha(\delta)\sim \delta^\frac{2(1-\mu)}{2-\mu}= \delta^\frac{1}{p+1},
\ee
which coincides with (\ref{Hoelderrate}) for all $0<p \le 1/2$.

Secondly, consider an accretive $A$ which is not necessarily self-adjoint, and let for some $0<p \le 1$ the source condition
 \be{eq:sc2}
 \xdag=(A^*)^p\, w,\quad w \in X,
 \ee
 hold. According to \cite[Proposition~7.0.1 (e)]{Haase06}, one has  $(A^*)^p=(A^p)^*$ {and} the inequality chain
\begin{eqnarray*}
\langle\xdag,x\rangle &=&\langle (A^*)^p w,  x\rangle
= \langle w, A^p x\rangle\\
&\leq& \norm{w}\,\norm{A^p x}
\le {c}\,\norm{w}\, \norm{x}^{1-p} \|Ax\|^{p},
\end{eqnarray*}
for some $c>0$,
{which is a consequence of the moment resp.~interpolation inequality for monotone operators, see \cite[Corollary 1.1.19]{Plato95} or \cite[Proposition 6.6.4]{Haase06}.
Then} Young's inequality yields, for all $0<p \le 1$, the variational source condition
$$\langle\xdag,x\rangle  \le \left( \frac{1-p}{2} \right)\|x\|^2+\left(\frac{1+p}{2}\right){c^{\frac{2}{p+1}}}\,\|w\|^\frac{2}{p+1}\,\|Ax\|^{\frac{2p}{p+1}}
$$
of type \eqref{vsc} occurring in Subsection~\ref{subsec:rates_varTikh} with $\psi(t)=\frac{2p}{p+1}$ and valid for all $x \in X$.
This gives for $0<p \le 1$
\be{eq:rate2}
\|x_{\alpha(\delta)}^\delta-\xdag\|=O(\delta^\frac{p}{2p+1}) \qquad
\mbox{if} \quad \alpha(\delta)\sim  \delta^\frac{p+1}{2p+1}.
\ee

{From \cite[Theorem~1 and Lemma~1]{PlMaHo16} it can be seen that under the source condition (\ref{eq:sc2}) the obtained rate (\ref{eq:rate2}) is not optimal for $0<p<1/2$ and can be improved to $O(\delta^{p/(p+1)})$ as in (\ref{eq:rate1}), because $\mathcal{R}(A^p)=\mathcal{R}((A^*)^p)$ for $0<p<\frac{1}{2}$.
On the other hand, for $1/2 \le p \le 1$ the ranges $\mathcal{R}(A^p)$ and $\mathcal{R}((A^*)^p)$ are in general different for non-self-adjoint
accretive operators $A$ as the Volterra operator shows (see  Subsection~\ref{ex:sc}). From \cite{PlMaHo16} one can also see
that for $1/2 \le p \le 1$ and under the source condition (\ref{eq:sc2}) the rate (\ref{eq:rate2}) is only optimal in the special case $p=1$. 
}}\end{remark}

\section{Examples}\label{sec:examples}
\subsection{Applying Lavrentiev-type variational source condition to linear problems with Volterra operator}\label{ex:sc}

{This subsection provides an example where a Lavrentiev-type variational source condition does provide the desired convergence rate.
We consider $X=L^2(0,1)$ and  the linear Volterra operator (simple integration operator) $A:X\to X$, defined by
$$[Ax](s):=\int_0^sx(t)\,dt,\quad 0 \le s \le 1,$$
which is accretive (monotone), but not self-adjoint.
For our study let us focus on the specific solution $\xdag\equiv 1$. This violates the source condition $\xdag=A^*v$  required for the rate $\norm{\xad-\xdag}=O\left(\delta^\frac13\right)$  in \cite{LiuNash96}, as for this purpose $\xdag$ would have to vanish at the right boundary point.  Moreover, $\xdag\equiv 1$ does not fulfill  $\xdag=A^{\frac12}v$ which would imply the same estimate along the lines of \cite{Taut02}. In fact,  we have that $\xdag\equiv 1$ satisfies $\xdag\in\mathcal{R}(A^p)=\mathcal{R}((A^*)^p)$ for any  $p\in(0,1/2)$, but we have $\xdag \notin \mathcal{R}(A^p)$
and $\xdag \notin \mathcal{R}((A^*)^p)$ for $p\in[1/2,1]$. This can be derived from the explicit structure of the ranges of the fractional powers of the Volterra operator given in \cite{GLY15}, and one can simply see that the corresponding ranges $\mathcal{R}(A^p)$ and $\mathcal{R}((A^*)^p)$ are different for all $1/2 \le p \le 1$.}
{Indeed, according to \cite[p.234, Prop. 8.5.5]{Haase06} we can rewrite $\xdag=A^pv$ as the Abel-type integral equation
\[
\xdag(s)=(A^p v) (s)=\frac{1}{\Gamma(p)}\int_0^s (s-t)^{p-1} v(t)\, dt\,,
\quad 0 \le s \le 1.
\]
Due to the identity $\Gamma(p)\Gamma(1-p)=\frac{\pi}{\sin(p\pi)}$ the well-known explicit solution formula for the Abel integral equation attains the form
\[
v(t)=\frac{1}{\Gamma(1-p)} \frac{d}{dt}\int_0^t(t-s)^{-p}\xdag(s)\, ds\,,
\quad 0 \le t \le 1,
\]
which yields $v(t)= \frac{1}{\Gamma(1-p)} t^{-p}$ for $\xdag \equiv 1$ so that $v\in L^2(0,1)$ iff $p<\frac12$.
}

However, a variational source condition \eqref{lvsc} resp.~\eqref{lvsc_gen} is satisfied, as shown below.

Let $z(s):=[Ax](s),\;0 \le s \le 1$. Integration by parts yields
\begin{eqnarray*}
\langle Ax,x\rangle=\int_0^1z(t)x(t)\,dt &=&\int_0^1z(t)\,dz(t)\\
&=&(z(1))^2-\int_0^1z(t)\,dz(t)=(z(1))^2-\langle Ax,x\rangle.
\end{eqnarray*}
This implies for all $x \in X$
$$\langle Ax,x\rangle=\frac{1}{2}\left(\int_0^1x(t)\,dt\right)^2$$
and hence we have for $\xdag\equiv 1$ and all $x \in X$
$$\langle\xdag,x\rangle=\int_0^1x(t)\,dt\leq\sqrt{2}\langle Ax,x\rangle^\frac{1}{2},$$
thus \eqref{lvsc} holds with $\beta_1=0,\; \beta_2=\sqrt{2}$  and  $\mu=\frac{1}{2}$. Then Theorem~\ref{th:rates_vi}  implies the convergence rate
$$\|\xad-\xdag\|=O\left(\delta^\frac{1}{3}\right)\,\,\,\,\,\mbox{for}\,\,\,\alpha\sim \delta^\frac{2}{3}.$$
The same result is also a consequence of Theorem \ref{th:rates} using approximate source conditions. Namely,
we have $d(R)\sim\frac{1}{R}$ for the distance function \linebreak $d(R)=\inf\{\norm{\xdag-Aw}: \norm{w}\leq R\}$, which immediately follows from the arguments of \cite[Example~5]{FHM11} by replacing there the function
$\sqrt{2}\cos(i-1/2)\pi t)$ with $\sqrt{2}\sin(i-1/2)\pi t)$ in order to execute the transfer from $A^*$ to $A$ in the eigenfunctions of the singular system.
Then we have from Theorem \ref{th:rates}
$$\|\xad-\xdag\|=O\left(\chi^{-1}(\Phi^{-1}(\delta))\right)=O\left(\delta^\frac{1}{3}\right).$$

\subsection{A parameter identification problem  in an elliptic PDE satisfying the local range invariance condition}

Consider identification of the source term $q$ in the elliptic boundary value problem
\begin{equation}\label{xi-example}
\begin{aligned}
-\Delta u+\xi(u)&=q \mbox{ in }\Omega\\
u&=0 \mbox{ on }\partial\Omega
\end{aligned}
\end{equation}
from measurements of $u$ in $\Omega$, where $\xi:\R\to\R$ is some Lipschitz continuously differentiable monotonically increasing function
and $\Omega \subseteq\R^3$ a smooth domain.
The corresponding forward operator $F:L^2(\Omega)\to H^2(\Omega)\subseteq L^2(\Omega)$, $q\mapsto u$, is indeed monotone, since
\be{Fximon}
\begin{aligned}
\langle F(q)-F(\tilde{q}),q-\tilde{q}\rangle
&=\int_\Omega (u-\tilde{u})(q-\tilde{q})\, dx\\
&=\int_\Omega (u-\tilde{u})\Bigl(-\Delta (u-\tilde{u})+\xi(u)-\xi(\tilde{u})\Bigr)\, dx\\
&=\int_\Omega \Bigl(|\nabla(u-\tilde{u})|^2+(\xi(u)-\xi(\tilde{u}))(u-\tilde{u})\Bigr)\, dx\\
&\geq \norm{\nabla(u-\tilde{u})}_{L^2(\Omega)}^2\geq0
\end{aligned}
\ee
(which by the way still does not imply strong monotonicity of $F$, since this would require a lower bound in terms of $\norm{q-\tilde{q}}_{L^2(\Omega)}^2$).
Lipschitz continuity of $F$ follows from the fact that $w=F(q)-F(q^\dagger)$ can be written as the solution of
\begin{equation}\label{PDEw}
\begin{aligned}
-\Delta w+\tilde{\xi}_q w&=q-q^\dagger \mbox{ in }\Omega\\
w&=0 \mbox{ on }\partial\Omega
\end{aligned}
\end{equation}
where 
\begin{equation}\label{xitil}
\tilde{\xi}_q =\int_0^1 \xi'(u^\dagger+t(F(q)-u^\dagger))\, dt\geq 0\,
\end{equation}
with $u^\dagger=F(q^\dagger)$.
Indeed, by testing with $w$, integration by parts and using Poincar\'{e}-Friedrichs' as well as Cauchy-Schwarz inequalities we get
\[
\begin{aligned}
&\frac{1}{C_{PF}}\|w\|_{L^2(\Omega)}^2\leq\|\nabla w\|_{L^2(\Omega)}^2
\leq \int_\Omega\Bigl( |\nabla w|^2+\tilde{\xi}_q \,w^2\Bigr)\, dx \\
&=\int_\Omega w(q-q^\dagger)\, dx
\leq \|w\|_{L^2(\Omega)} \|q-q^\dagger\|_{L^2(\Omega)} \,.
\end{aligned}
\]
According to Theorem 1.4.6 in \cite{AlbRya06}, the operator $F$ is actually maximal monotone, since it is monotone, continuous  (hence hemicontinuous) and its domain is the space $L^2(\Omega)$.
We define the linear operator $A:L^2(\Omega)\to L^2(\Omega)$ as $h\mapsto Ah=v$ with $v$ solving
\begin{equation}\label{PDEv}
\begin{aligned}
-\Delta v+\xi'(u^\dagger)v&=h \mbox{ in }\Omega\\
v&=0 \mbox{ on }\partial\Omega
\end{aligned}
\end{equation}
where $u^\dagger$ solves \req{xi-example} with $q=q^\dagger$.

We claim that $F$ satisfies the range invariance condition \req{invariance} with $N\equiv0$.  In order to show this,  we use the representation of $w=F(q)-F(q^\dagger)$ \req{PDEw}, \req{xitil} from above, as well as \req{PDEv}, which yields
that $F(q)-F(q^\dagger)=A M(q)(q-q^\dagger)$ holds with
\[
M(q):L^2(\Omega)\to L^2(\Omega)\,, \quad
M(q)=(-\Delta+\xi'(u^\dagger)\mbox{ id})(-\Delta+\tilde{\xi}_q \mbox{ id})^{-1}\,.
\]
Note that for any $a\in L^2(\Omega)$, $a\geq0$ a.e., the linear operator
\[
(-\Delta+a \mbox{ id}):H^2(\Omega)\cap H_0^1(\Omega)\to L^2(\Omega)\,, \quad
u\mapsto -\Delta u + a u
\]
is well-defined and continuously invertible and both $A$ and $AM(q)$ are accretive since for any $a\in L^2(\Omega)$, $a\geq0$ a.e., (so e.g., $a=\xi'(u^\dagger)$ or $a=\tilde{\xi}_q$), $h\in L^2(\Omega)$ and $v:=(-\Delta+a \mbox{ id})^{-1} h$, we have
\[
\langle (-\Delta+a \mbox{ id})^{-1} h, h\rangle =\int_\Omega (|\nabla v|^2+av^2)\, dx\geq0.
\]
The difference of $M(q)$ to the identity can be estimated as follows:
\[
\begin{aligned}
\norm{M(q)-I}&=\norm{
\Bigl((-\Delta+\xi'(u^\dagger)\mbox{ id})-(-\Delta+\tilde{\xi}_q \mbox{ id})\Bigr)
(-\Delta+\tilde{\xi}_q \mbox{ id})^{-1}}\\
&=\norm{
(\xi'(u^\dagger)-\tilde{\xi}_q )\mbox{ id}
(-\Delta+\tilde{\xi}_q \mbox{ id})^{-1}}\\
&=\sup_{0\not=f\in L^2(\Omega)} \frac{\norm{(\xi'(u^\dagger)-\tilde{\xi}_q )
(-\Delta+\tilde{\xi}_q \mbox{ id})^{-1}f}_{L^2}}{\norm{f}_{L^2}}\\
&=\sup_{0\not=w\in H^2(\Omega)\cap H_0^1(\Omega)} \frac{\norm{(\xi'(u^\dagger)-\tilde{\xi}_q )w}_{L^2}}{\norm{-\Delta w+\tilde{\xi}_q  w}_{L^2}}\\
&\leq C_\Delta \sup_{0\not=w\in H^2(\Omega)\cap H_0^1(\Omega)} \frac{\norm{\xi'(u^\dagger)-\tilde{\xi}_q }_{L^2}\norm{w}_{L^2}}{\norm{w}_{H^2}}\\
&\leq C_\Delta C_{H^2\to L^\infty}^\Omega \norm{\xi'(u^\dagger)-\tilde{\xi}_q }_{L^2}\,,
\end{aligned}
\]
where we have used elliptic regularity, i.e., the above mentioned mapping properties of $(-\Delta+a \mbox{ id})$, and continuity of the embedding $H^2(\Omega)\to L^\infty(\Omega)$ with norm $C_{H^2\to L^\infty}^\Omega$.
Here one has
\[
\begin{aligned}
\norm{\xi'(u^\dagger)-\tilde{\xi}_q }_{L^2}^2
&=\int_\Omega \Bigl(\int_0^1\xi'(u^\dagger)- \xi'(u^\dagger+t(u-u^\dagger))\, dt\Bigr)^2\, dx\\
&\leq L_{\xi'}^2\int_\Omega \Bigl(\int_0^1 t(u-u^\dagger)\, dt\Bigr)^2\, dx\\
&\leq L_{\xi'}^2\int_\Omega \Bigl(\int_0^1 t^2\,dt
\int_0^1(u-u^\dagger)^2\, dt\Bigr)\, dx\\
&=\frac{L_{\xi'}^2}{3}\norm{u-u^\dagger}_{L^2}^2
\leq\frac{L_{\xi'}^2L_{PF}^2}{3}\norm{q-q^\dagger}_{L^2}^2,
\,
\end{aligned}
\]
for some constant $ L_{\xi'}>0.$
Consequently, we obtain
\[
\norm{M(q)-I}\leq C_\Delta C_{H^2\to L^\infty}^\Omega
\frac{L_{\xi'}L_{PF}}{\sqrt{3}}\norm{q-q^\dagger}_{L^2}
\leq c
\]
for all $q\in \mathcal{B}_\rho(q^\dagger)$,
where $c=C_\Delta C_{H^2\to L^\infty}^\Omega
\frac{L_{\xi'}L_{PF}}{\sqrt{3}}\rho$ is smaller than one, provided that $\rho$ is sufficiently small.

\medskip

{
We now verify the Lavrentiev specific source condition \req{lvsc_gen} 
for the special case of H\"older rates $\varphi(t)\sim t^\mu$, $\mu\in[0,\frac12]$. More precisely, we will show that 
\[
q^\dagger-\bar{q}\in H_0^{2\mu}(\Omega) 
\]
(which in case $\mu<\frac14$ is equivalent to $q^\dagger-\bar{q}\in H^{2\mu}(\Omega)$, hence does not impose any boundary conditions on $q^\dagger-\bar{q}$)
is sufficient for this variational source condition, provided the domain is sufficiently smooth to admit the elliptic regularity and embedding bounds
\[
\begin{aligned}
&\norm{(-\Delta)^{\mu}}_{H^{2\mu}(\Omega)\to L^2(\Omega)}=:C^\Delta_{2\mu}<\infty\,, \
&&\norm{(-\Delta)^{1-\mu}}_{H^{2-2\mu}(\Omega)\to L^2(\Omega)}=:C^\Delta_{2-2\mu}<\infty\,,\\
&\norm{(-\Delta)^{-1}}_{L^2(\Omega)\to H^2(\Omega)}=:C^\Delta<\infty\,,\ 
&&\norm{\mbox{id}}_{H^{-2\mu}(\Omega)\to L^2(\Omega)}=:C^\Omega_{2\mu}<\infty\,,
\end{aligned}
\]
where $-\Delta$ is the Laplace operator equipped with homogeneous Dirichlet boundary conditions.
For this purpose we use the fact that by \req{Fximon}, for any $q\in\mathcal{M}:=\mathcal{B}_\rho(q^\dagger)$ we have 
\[
\gamma:=\langle F(q)-F(q^\dagger),q-q^\dagger\rangle\geq \norm{\nabla(u-u^\dagger)}_{L^2(\Omega)}^2
\geq \frac{1}{1+C_{PF}} \norm{u-u^\dagger}_{H^1(\Omega)}^2
\]
as well as 
\[
\begin{aligned}
&\norm{\xi(u)-\xi(u^\dagger)}_{H^{-2\mu}(\Omega)}
\leq C^\Omega_{2\mu} \norm{\xi(u)-\xi(u^\dagger)}_{L^2(\Omega)}
= C^\Omega_{2\mu} \sqrt{\int_\Omega (\xi(u)-\xi(u^\dagger))^2\, dx}\\
&\leq C^\Omega_{2\mu} \sqrt{\int_\Omega L_\xi(\xi(u)-\xi(u^\dagger))(u-u^\dagger)\, dx}
\leq C^\Omega_{2\mu} \sqrt{L_\xi} \sqrt{\gamma}\,.
\end{aligned}
\]
The first of these two estimates can be further made use of in the interpolation estimate
\[
\begin{aligned}
&\norm{-\Delta (u^\dagger-u)}_{H^{-2\mu}(\Omega)}
=\sup_{v\in C_0^\infty(\Omega)\setminus\{0\}}\frac{\langle (-\Delta) (u^\dagger-u),v\rangle}{\norm{v}_{H^{2\mu}(\Omega)}}\\
&=\sup_{v\in C_0^\infty(\Omega)\setminus\{0\}}\frac{\langle (-\Delta)^{1-\mu} (u^\dagger-u),(-\Delta)^{\mu}v\rangle}{\norm{v}_{H^{2\mu}(\Omega)}}\\
&\leq C^\Delta_{2\mu} C^\Delta_{2-2\mu} \norm{u^\dagger-u}_{H^{2-2\mu}(\Omega)}\\
&\leq C^\Delta_{2\mu} C^\Delta_{2-2\mu} \norm{u^\dagger-u}_{H^2(\Omega)}^{1-2\mu} \norm{u^\dagger-u}_{H^1(\Omega)}^{2\mu}\\
&\leq C^\Delta_{2\mu} C^\Delta_{2-2\mu} (C^\Delta \norm{-\Delta(u^\dagger-u)}_{L^2(\Omega)})^{1-2\mu} 
((1+C_{PF})\gamma)^{\mu}\,,
\end{aligned}
\]
where we can further estimate 
\[
\norm{-\Delta(u^\dagger-u)}_{L^2(\Omega)}
=\norm{q^\dagger-q -(\xi(u^\dagger)-\xi(u))}_{L^2(\Omega)}
\leq \rho + \sqrt{L_\xi} \sqrt{\gamma}\,.
\]
This altogether yields
\[
\begin{aligned}
&\langle q^\dagger-\bar q,q^\dagger-q\rangle 
= \langle q^\dagger-\bar q, -\Delta (u^\dagger-u) +\xi(u^\dagger)-\xi(u)\rangle\\
&\leq \norm{q^\dagger-\bar{q}}_{H_0^{2\mu}(\Omega)} \Bigl(
\norm{-\Delta (u^\dagger-u)}_{H^{-2\mu}(\Omega)}+\norm{\xi(u)-\xi(u^\dagger)}_{H^{-2\mu}(\Omega)}\Bigr)\\
&\leq \norm{q^\dagger-\bar{q}}_{H_0^{2\mu}(\Omega)} \\
&\qquad\cdot\Bigl(
C^\Delta_{2\mu} C^\Delta_{2-2\mu} (C^\Delta (\rho + \sqrt{L_\xi} \sqrt{\gamma}))^{1-2\mu} 
((1+C_{PF})\gamma)^{\mu}
+C^\Omega_{2\mu} \sqrt{L_\xi} \sqrt{\gamma}
\Bigr)
\\
&=:\varphi(\gamma)=O(\gamma^\mu).
\end{aligned}
\]
}

\section{Conclusions and further work}
In this paper, we have shown convergence rates for Lavrentiev's regularization method under variational and approximate source conditions for linear and nonlinear inverse problems with monotone forward operators. In particular, we have proposed a new variational source condition that seems to be quite appropriate for the Lavrentiev setting.

To compare the capability of the different source conditions, we present in the following table the best possible rates in the linear case.
\\[1ex]
\begin{tabular}{l|l|l}
& condition & rate\\
\hline
(a)& $\xdag=Aw$ for some $w\in X$ & $O(\delta^{\frac12})$\\
(b)& $\xdag=A^*w$ for some $w\in X$ & $O(\delta^{\frac13})$\\
(c)& $\langle\xdag,x\rangle\leq \beta_1\norm{x}^2+\beta_2\langle A x,x\rangle^{\frac{1}{2}}$ for all $x\in X$ & $O(\delta^{\frac13})$
\end{tabular}
\\[1ex]
As the example from Section~\ref{ex:sc} shows, (c) implies neither (a) nor (b).
However, (c) is implied by a fractional source condition $\xdag=A^{\frac12}w$ in case of a self-adjoint operator $A$, cf. Remark~\ref{rem:scexvar}.
{
The question of rates beyond those stated above appears to be a challenging one, which we intend to investigate further on.}

Note that  no additional restriction on the nonlinearity of $F$ is needed as regards  variational source conditions-- not even differentiability -- as already observed in \cite{HKPS07} for Tikhonov regularization. On the other hand, we had to impose some local range invariance condition in order to prove rates with approximate source conditions.

Lavrentiev's method in Banach spaces will be subject of further research in light of the few aspects considered in this Hilbertian setting.

\section*{Acknowledgments}
{
We wish to thank Peter Math\'{e} (WIAS, Berlin), Robert Plato (University of Siegen) and R.~I.~Bo\c{t} (University of Vienna) for  valuable discussions.\\
The research of the first author was partially supported by the German Research Foundation (DFG) under grant HO~1454/8-2.
The second author acknowledges support by the Austrian Science Fund FWF under grand I2271.
The second and third author were supported by the Karl Popper Kolleg ``Modeling-Simulation-Optimization'' funded by the Alpen-Adria-Universit\"at Klagenfurt and by the Carin\-thian Economic Promotion Fund (KWF)
\\
Moreover, we wish to thank both reviewers for fruitful comments leading to an improved version of the manuscript.
}

\end{document}